\documentclass{amsart}

\usepackage{amssymb}
\usepackage{amsthm}
\usepackage{amsmath}
\usepackage{graphicx}
\usepackage{pifont}
\usepackage{txfonts}
\usepackage{color}

\theoremstyle{plain}
\newtheorem{theorem}{Theorem}[section]
\newtheorem{proposition}[theorem]{Proposition}%[section]
\newtheorem{corollary}[theorem]{Corollary}%[section]
\newtheorem{lemma}[theorem]{Lemma}%[section]

\newtheorem*{Theorem}{Theorem}
\newtheorem*{Corollary}{Corollary}

\newtheorem*{conjecture*}{Conjecture}
\theoremstyle{definition}
\newtheorem{definition}[theorem]{Definition}
\newtheorem*{definition*}{Definition}%[section]
\newtheorem*{example*}{Example}
\newtheorem*{notation*}{Notation}
\newtheorem*{notation-conv*}{Notation and convention}
\newtheorem*{convention*}{Convention}
\theoremstyle{remark}
\newtheorem{remark}[theorem]{Remark}%[sectio]
\newtheorem*{remark*}{Remark}%[sectio]

\newcommand{\Z}{\mathbb{Z}}

\newcommand{\R}{\mathbb{R}}
\newcommand{\C}{\mathbb{C}}

\newcommand{\upperH}{\mathbb{H}}

\newcommand{\SLC}[1][2]{\mathrm{SL}_{#1}(\C)}
\newcommand{\PSLC}{\mathrm{PSL}_{2}(\C)}
\newcommand{\SLR}[1][2]{\mathrm{SL}_{#1}(\R)}
\newcommand{\TSLR}{\widetilde{\mathrm{SL}}_{2}(\R)}
\newcommand{\TPSLR}{\widetilde{\mathrm{PSL}}_{2}(\R)}
\newcommand{\PSLR}{\mathrm{PSL}_{2}(\R)}
\newcommand{\SU}{\mathrm{SU}(2)}
\newcommand{\SUII}{\mathrm{SU}(1,1)}

\newcommand{\trho}{\widetilde{\rho}}
\newcommand{\shift}{\mathrm{sh}}

\newcommand{\trace}{{\rm tr}\,}
\newcommand{\I}{I}

\newcommand{\mfd}{M}
\newcommand{\Seifert}{M\hbox{$\big(\frac{1}{b}, \frac{\alpha_1}{\beta_1}, \ldots, \frac{\alpha_n}{\beta_n}\big)$}}

\newcommand{\univcover}[1]{\widetilde{#1}}

\newcommand{\Hom}{\mathop{\mathrm{Hom}}}

\newcommand{\Ext}{\mathop{\mathrm{Ext}}}
\newcommand{\Isom}{\mathop{\mathrm{Isom}}}

\newcommand{\Tor}[2]{\mathop{\mathrm{Tor}}\nolimits (#1;#2)}
\newcommand{\ie}{i.e.,\,}
\begin{document}

%%%%%%%%%%%%%%%%%%%%%%%%%%%%%%%%%%%%%%%%%% 
% title 
%%%%%%%%%%%%%%%%%%%%%%%%%%%%%%%%%%%%%%%%%% 

\title[Asymptotics of R-torsion for Seifert manifolds and $\PSLR$-reps. of Fuchsian groups]{
  The asymptotic behavior of the Reidemeister torsion for Seifert manifolds and $\PSLR$-representations of Fuchsian groups
}

%%%%%%%%%%%%%%%%%%%%%%%%%%%%%%%%%%%%%%%%%% 
% author names and addresses
%%%%%%%%%%%%%%%%%%%%%%%%%%%%%%%%%%%%%%%%%% 
\author{Yoshikazu Yamaguchi}

\address{Department of Mathematics,
  Akita University,
  1-1 Tegata-Gakuenmachi, Akita, 010-8502, Japan}
\email{shouji@math.akita-u.ac.jp}

%% \date{\today}

%%%%%%%%%%%%%%%%%%%%%%%%%%%%%%%%%%%%%%%%%% 
% Subject classification 
%%%%%%%%%%%%%%%%%%%%%%%%%%%%%%%%%%%%%%%%%% 

%%%%%%%%%%%%%%%%%%%%%%%%%%%%%%%%%%%%%%%%%% 
% abstract
%%%%%%%%%%%%%%%%%%%%%%%%%%%%%%%%%%%%%%%%%% 
\begin{abstract}
  We show that a $\PSLR$-representation of a Fuchsian group
  induces the asymptotics of the Reidemeister torsion
  for the Seifert manifold corresponding to the euler class of the
  $\PSLR$-representation.
  We also show that the limit of leading coefficient of
  the Reidemeister torsion is determined by the euler class
  of a $\PSLR$-representation of a Fuchsian group.
  In particular, 
  the leading coefficient of the Reidemeister torsion for
  the unit tangent bundle over a two--orbifold converges to $-\chi \log2$
  where $\chi$ is the Euler characteristic of the two--orbifold.
  We also give a relation between $\Z_2$-extensions for
  $\PSLR$-representations of a Fuchsian group and
  the asymptotics of the Reidemeister torsion.
\end{abstract}

%%%%%%%%%%%%%%%%%%%%%%%%%%%%%%%%%%%%%%%%%%%%%%%%%%%%%%%%%%%%%%%%%%%% 
% body of paper
%%%%%%%%%%%%%%%%%%%%%%%%%%%%%%%%%%%%%%%%%%%%%%%%%%%%%%%%%%%%%%%%%%%% 

\maketitle
%% \tableofcontents

\section{Introduction}
The previous works~\cite{yamaguchi:RtorTorusKnots, Yamaguchi:asymptoticsRtorsion}
of the author have investigated the asymptotic behavior of the Reidemeister torsion 
for a Seifert manifold with an sequence of $\SLC[2N]$-representation of the fundamental group.
Here the sequence of $\SLC[2N]$-representations starts with an $\SLC$-representation
and the remaining representations are given by 
the composition with the irreducible $2N$-dimensional representations of $\SLC$.
A sequence of $\SLC[2N]$-representations defines a sequence of the Reidemeister torsion of
a Seifert manifold. We can consider the asymptotic behavior of the sequence 
given by the Reidemeister torsions.
Under the constraint that $\SLC[2N]$-representations send a regular fiber to $-\I_{2N}$,
the observation in~\cite{Yamaguchi:asymptoticsRtorsion}
revealed the growth order and the convergence of the leading coefficient
of the Reidemeister torsion for a Seifert manifold.

In this paper, we observe when we have a natural situation 
for a Seifert manifold 
and $\SLR$-representations satisfying our constraint.
We will deal with closed Seifert manifolds
whose base orbifolds have the negative Euler characteristics.
Namely, they admit the $\upperH^2 \times \R$ or $\TSLR$-geometry.
We can regard
the fundamental group of the base orbifold as a cocompact Fuchsian group $\Gamma$ 
which can be embedded in $\PSLR$.
The fundamental group of a Seifert manifold is a central extension of the Fuchsian group $\Gamma$
by $\Z$.
The universal cover $\TPSLR$ is also a central extension of $\PSLR$ by $\Z$.
If we choose a homomorphism $\bar\rho$ from $\Gamma$ to $\PSLR$ for a Seifert manifold $\mfd$,
then we have a lift $\trho$ from $\pi_1(\mfd)$ to $\TPSLR$.
This is due to the work~\cite{JankinsNeumannHomomorphisms} by M.~Jankins and W.~Neumann
(we review this in Section~\ref{sec:rep_Fuchsian}).
Since $\TPSLR$ is also the universal cover of $\SLR$,
we have the $\SLR$-representation $\rho$ of $\pi_1(\mfd)$ given by the composition of $\trho$
with the projection from $\TPSLR$ onto $\SLR$.
The $\SLR$-representation $\rho$ induces the $\SLR[2N]$-representations which define
the acyclic chain complexes of the Seifert manifold $\mfd$.
We can observe the asymptotics of the Reidemeister torsion 
starting with a $\PSLR$-representation of a Fuchsian group.

We refer the case of a Seifert homology sphere and $\SU$-representations
to the previous work~\cite{Yamaguchi:asymptoticsRtorsion}, in which 
the maximum and minimum of the limits have been observed in detail.
The limit of the leading coefficient is determined by
the images of exceptional fibers by a given $\SU$-representation.

First, we will show that 
the limit of leading coefficient in the asymptotic behavior of the Reidemeister torsions
is determined by a representation of a Fuchsian group. 
Here we start with a $\PSLR$-representation of the Fuchsian group.
Let $\bar\rho$ be a $\PSLR$-representation of a cocompact Fuchsian group
$\Gamma = \Gamma(g; \alpha_1, \ldots, \alpha_n)$
(for the notation, see Section~\ref{sec:rep_Fuchsian}).
If we have the following diagram:
\[
\setlength{\arraycolsep}{2pt}
\begin{array}{ccccccccc}
  0 & \to & \Z & \to &\pi_1(\Seifert)& \to & \Gamma & \to & 1 \\
  &     & || &   & \quad \downarrow \raisebox{1pt}{$\scriptstyle \widetilde{\rho}$} & & \quad \downarrow \raisebox{2pt}{$\scriptstyle \bar\rho$} &  \\
  0 & \to & \Z & \to & \TPSLR & \to & \PSLR & \to & 1,
\end{array}
\]
then we have the $\SLR$-representation $\rho$ of $\pi_1(\Seifert)$
which induces a sequence of the Reidemeister torsion $\Tor{\mfd}{\rho_{2N}}$
(for the details on this sequence,
we refer to Theorem~\ref{thm:asymptotics_Rtorsion} and Lemma~\ref{lemma:h_to_shift_1}).
The sequence of $\log|\Tor{\Seifert}{\rho_{2N}}$ has the growth order of $2N$ and the following leading coefficient.

\begin{Theorem}[Theorem~\ref{thm:main_I}]
  For the induced $\SLR[2N]$-representations $\rho_{2N}$, 
  the limit of the leading coefficient of $\log |\Tor{\Seifert}{\rho_{2N}}|$
  is expressed as 
  \begin{equation*}
    \lim_{N \to \infty}
    \frac{\log|\Tor{\Seifert}{\rho_{2N}}|
    }{2N}
    = 
    - \Big(2 -2g - \sum_{j=1}^{n} \frac{\lambda_j - 1}{\lambda_j}\Big) \log 2.
  \end{equation*}
  %% where $\lambda_j$ is the quotient of $\alpha_j$ by the g.c.d. $(\alpha_j, \beta_j)$.
  Here $\lambda_j$ is the order of $\bar \rho (q_j)$ where $q_j$ is the homotopy class
  of a loop around the $j$-th branched point on the two--orbifold.
\end{Theorem}

\begin{remark*}
  The existence of a lift $\widetilde{\rho}$ is determined by
  the euler class $e(\bar\rho)$ in $H^2(\Gamma;\Z)$.
  The explicit value of $\lambda_j$ is given in Theorem~\ref{thm:main_I}.
  The limit of the leading coefficient is also determined by the euler class $e(\bar\rho)$.
\end{remark*}

The unit tangent bundle 
over a two--orbifold $\upperH^2 / \Gamma$
is also a Seifert manifold with the index
$((1, 2g-2), (\alpha_1, \alpha_1 -1), \ldots, (\alpha_n, \alpha_n -1))$.
We can think of the unit tangent bundle $T^1 (\upperH^2 / \Gamma)$ as $\PSLR / \Gamma$.
For a Seifert manifold $\PSLR / \Gamma$, 
we can take a $\TPSLR$-representation $\trho$ of $\pi_1(\PSLR / \Gamma)$ as 
a lift of an embedding from $\Gamma$ into $\PSLR$.

\begin{Corollary}[Corollary~\ref{cor:unit_tangent_bundle}]
  Suppose that an $\SLR$-representation $\rho$ of $\pi_1(\PSLR / \Gamma)$ is the composition of $\trho$
  with the projection from $\TPSLR$ onto $\SLR$.
  Then we have the following limit of the leading coefficient of
  $\log |\Tor{\PSLR / \Gamma}{\rho_{2N}}|$:
  \begin{align*}
    \lim_{N \to \infty}
    \frac{\log|\Tor{\PSLR / \Gamma}{\rho_{2N}}|
    }{2N}
    &= 
    - \Big(2 -2g - \sum_{j=1}^{n} \frac{\alpha_j - 1}{\alpha_j}\Big) \log 2\\
    &= - \chi \log 2
  \end{align*}
  where $\chi$ is the Euler characteristic of the base orbifold $\upperH^2 / \Gamma$.
\end{Corollary}

We also investigate which $\PSLR$-representation $\bar\rho$ of $\Gamma$ induces
$\SLR$-representation of $\pi_1(\Seifert)$.
We show a sufficient condition for $\bar\rho$ to be lifted 
to an $\SLR$-representation for a given Seifert manifold $\Seifert$
in the following diagram:
\[
\setlength{\arraycolsep}{2pt}
\begin{array}{ccccccccc}
  0 & \to & \Z & \to & \pi_1(\Seifert) & \to & \Gamma & \to & 1 \\
  &     & \downarrow &   & \; \downarrow \raisebox{1pt}{${\scriptstyle \rho}$} & & \; \downarrow \raisebox{2pt}{$\scriptstyle \bar\rho$} &  \\
  0 & \to & \Z / 2\Z& \to & \SLR & \to & \PSLR & \to & 1.
\end{array}
\]
Our sufficient condition is given 
in terms of the euler class of $\bar\rho$.
\begin{Theorem}[Theorem~\ref{prop:Z2extension_representation}]
  A $\PSLR$-representation $\bar\rho$ of $\Gamma$
  induces an $\SLR$-one of $\pi_1(\Seifert)$ such that $\rho(h)= -\I$
  if
  the euler class $e(\bar\rho)$ satisfies the criteria of
  Theorem~\ref{thm:criterion_JN}
  and gives the equivalent class
  $[b x_0 + \beta_1 x_1 + \cdots + \beta_n x_n]$ in $\Ext(\Gamma;\Z/2\Z)$.
\end{Theorem}
Moreover the limit of the leading coefficient for $\Tor(\Seifert;\rho_{2N})$
is expressed in Theorem~\ref{thm:asymptotic_another_SLR_rep}.
We will compute explicit examples for $\SLR$-representations of
Brieskorn manifolds by using Theorem~\ref{prop:Z2extension_representation}

\subsection*{Organization}
We review the previous result on the asymptotic
behavior of the Reidemeister torsion in Section~\ref{sec:review_asymptotics}.
Section~\ref{sec:rep_Fuchsian} gives a brief review the work on 
$\PSLR$-representations of a cocompact Fuchsian group by 
Jankins and Neumann~\cite{JankinsNeumannHomomorphisms}.
In Section~\ref{sec:asymptotics_SL2R},
We observe the asymptotic behavior of the Reidemeister torsion
for a Seifert manifold and the $\SLR$-representation induced by
a $\TPSLR$-one.
We deal with $\SLR$-representations for a Seifert manifold,
which are given by $\PSLR$-representations of a Fuchsian group with different euler classes
in Section~\ref{sec:Z_2_extensions}.
The last Section~\ref{sec:examples} shows explicit examples of 
$\SLR$-representations induced by $\TPSLR$-ones for 
Brieskorn manifolds.

\section{Preliminaries}
\label{sec:preliminaries}
For a Seifert index $(g; (1, b), (\alpha_1, \beta_1), \ldots, (\alpha_n, \beta_n))$,
we follow the convention of Jankins and Neumann~\cite{JankinsNeumannHomomorphisms}.
This notation differs from that of~\cite{Yamaguchi:asymptoticsRtorsion}
in sign of $b$ and $\beta_1, \ldots, \beta_n$.

\subsection{Asymptotic behavior of the Reidemeister torsion}
\label{sec:review_asymptotics}
Let $\mfd$ be a Seifert manifold of the index
$(g; (1, b), (\alpha_1, \beta_1), \ldots, (\alpha_n, \beta_n))$.
The fundamental group of $\mfd$ is expressed as
\[
\pi_1(\mfd)
=
\langle
a_1, b_1, \ldots, a_g, b_g, q_1, \ldots, q_n, h \,|\,
\hbox{$h$ is central}, q_j^{\alpha_j} = h^{\beta_j}, q_1 \cdots q_n \prod_{i=1}^g [a_i, b_i] = h^{-b}
\rangle.
\]
We use the symbol $\rho$ for an $\SLC$-representation of $\pi_1(\mfd)$.
We denote by $\rho_{2N}$ the composition of $\rho$ with
the irreducible $2N$--dimensional representation of $\SLC$
and by $\Tor{\mfd}{\rho_{2N}}$ the Reidemeister torsion of $\mfd$ and $\rho_{2N}$. 
The asymptotic behavior of the Reidemeister torsion is expressed as follows.
\begin{theorem}[Theorem~$4.5$ in~\cite{Yamaguchi:asymptoticsRtorsion}]
  \label{thm:asymptotics_Rtorsion}
  If $\pi_1(\mfd)$ has an $\SLC$-representation $\rho$
  sending the homotopy class $h$ of a regular fiber to $-\I$,
  the asymptotic behavior of the sequence given by the Reidemeister torsion $\Tor{\mfd}{\rho_{2N}}$
  is expressed as 
  \begin{align}
    \lim_{N \to \infty}
    \frac{\log|\Tor{\mfd}{\rho_{2N}}|
    }{(2N)^2}
    &= 0,  \notag\\
    \label{eqn:asymptotic_Rtorsion}
    \lim_{N \to \infty}
    \frac{\log|\Tor{\mfd}{\rho_{2N}}|
    }{2N}
    &= 
    - \Big(2 -2g - \sum_{j=1}^{n} \frac{\lambda_j - 1}{\lambda_j}\Big) \log 2.
  \end{align}
  The right hand side of \eqref{eqn:asymptotic_Rtorsion} is determined by the orders of the $\SLC$-matrices for 
  the exceptional fibers. When we denote by $\ell_j$ the homotopy class of $j$-th exceptional fiber,
  each $\lambda_j$ is half the order of $\rho(\ell_j)$.
\end{theorem}

\subsection{$\PSLR$-representations of Fuchsian groups and the euler classes}
\label{sec:rep_Fuchsian}
We use the symbol $\Gamma = \Gamma(g; \alpha_1, \ldots, \alpha_n)$
for a cocompact Fuchsian group of genus $g$ with 
branch indices $\alpha_1, \ldots, \alpha_n$.
The Fuchsian group $\Gamma$ has the following presentation:
\begin{equation}
  \label{eqn:presentation_orbifold_fgroup}
  \Gamma =
  \langle
  a_1, b_1, \ldots, a_g, b_g, q_1, \ldots, q_n \,|\,
  q_j^{\alpha_j} = 1, q_1 \cdots q_n \prod_{i=1}^g [a_i, b_i] = 1
  \rangle.
\end{equation}
In~\cite{JankinsNeumannHomomorphisms}, Jankins and Neumann determined the set of components in $\Hom(\Gamma, \PSLR)$
by an euler class:
\[
e:\Hom(\Gamma, \PSLR) \to
H^2(\Gamma;\Z) = ab\langle x_0, \ldots, x_n \,|\, \alpha_i x_i = x_0, i=1, \ldots, n \rangle
\]
The Euler class $e$ is defined as follows.
Taking the pull--back central extension for $f \in \Hom(\Gamma, \PSLR)$ from 
\[
\setlength{\arraycolsep}{2pt}
\begin{array}{ccccccccc}
  & & & & & & \Gamma &  &  \\
  & & & & & & \quad \downarrow \raisebox{2pt}{$\scriptstyle f$} &  \\
  0 & \to & \Z & \to & \TPSLR & \to & \PSLR & \to & 1,
\end{array}
\]
we have the following commutative diagram:
\begin{equation}
  \setlength{\arraycolsep}{2pt}
  \begin{array}{ccccccccc}
    0 & \to & \Z & \to &\univcover{\Gamma}& \to & \Gamma & \to & 1 \\
    &     & || &   & \quad \downarrow \raisebox{2pt}{$\scriptstyle \widetilde{f}$} & & \quad \downarrow \raisebox{2pt}{$\scriptstyle f$} &  \\
    0 & \to & \Z & \to & \TPSLR & \to & \PSLR & \to & 1.
  \end{array}
\end{equation}
The central extension $\univcover{\Gamma}$ has a presentation:
\[\langle
a_1, b_1, \ldots, a_g, b_g, q_1, \ldots, q_n, h \,|\,
\hbox{$h$ is central}, q_j^{\alpha_j} = h^{\beta_j}, q_1 \cdots q_n \prod_{i=1}^g [a_i, b_i] = h^{-b}
\rangle\]
with $0 \leq \beta_i < \alpha_i$.
By the presentation of $\univcover{\Gamma}$,
we define $e(f)$ to be $b x_0 + \beta_1 x_1 + \cdots + \beta_n x_n$.

\begin{theorem}[Theorem~$1$ in~\cite{JankinsNeumannHomomorphisms}]
  \label{thm:criterion_JN}
  Suppose that $x \in H^2(\Gamma;\Z)$ satisfies   
  $x = bx_0 + \beta_1 x_1 + \cdots + \beta_n x_n$ ($0 \leq \beta_i < \alpha_i$).
  There exists some $f$ such that $e(f) = x$
  if and only if
  the following holds:
  \begin{enumerate}
  \item \label{item:g_geq_0}
    If $g > 0$ then $2 - 2g -n \leq b \leq 2g - 2$;
  \item \label{item:g_eq_0}
    If $g=0$ then either
    \begin{enumerate}
    \item $2 - n \leq b \leq -2$ or;
    \item \label{item:b_eq_minus_one}
      $b = -1$ and $\sum_{j=1}^n (\beta_j / \alpha_j) \leq 1$ or;
    \item \label{item:b_eq_one_minus_n}
      $b = 1 - n$ and $\sum_{j=1}^n (\beta_j / \alpha_j) \geq n-1$.
    \end{enumerate}
  \end{enumerate}
\end{theorem}

\begin{remark}
The above definition of the euler class $e$ is
the alternative one used in the proof of~\cite[Theorem~$1$]{JankinsNeumannHomomorphisms}.
The euler class $e$ is defined as $e(f) = f^*(c)$ where 
$f^* : H^2(B\PSLR;\Z) \to H^2(B\Gamma;\Z)=H^2(\Gamma;\Z)$ is the induced homomorphism by $f$
and $c$ is a generator of $H^2(B\PSLR;\Z) \simeq \Z$.
\end{remark}

Jankins and Neumann also proved the following theorem
on the components of the $\PSLR$-representation space
of a Fuchsian group.
\begin{theorem}[Theorem~$2$ in \cite{JankinsNeumannHomomorphisms}]
\label{thm:components_fuchsian}
Let $\Gamma$ be a cocompact Fuchsian group.
The fibers of $e$ are the components of $\Hom(\Gamma, \PSLR)$.
\end{theorem}

\section{Asymptotics of the Reidemeister torsion via $\TPSLR$-representations}
\label{sec:asymptotics_SL2R}
In the observation of the asymptotic behavior for the Reidemeister torsion,
we require that an $\SLC$-representation $\rho$ sends the central element $h$ to $-\I$.
Since $\pi_1(\mfd)$ is a central extension of $\Gamma$, 
this is equivalent to that we have the following commutative diagram:
\[
\setlength{\arraycolsep}{2pt}
\begin{array}{ccccccccc}
0 & \to & \Z & \to & \pi_1(\mfd) & \to & \Gamma & \to & 1 \\
  & & \downarrow & & \downarrow \rho & & \quad \downarrow \bar\rho &  \\
0 & \to & \{\pm I\} & \to & \SLC & \to & \PSLC & \to & 1.
\end{array}
\]

The limit~\eqref{eqn:asymptotic_Rtorsion} in Theorem~\ref{thm:asymptotics_Rtorsion}
is determined by the induced $\PSLC$-representation $\bar\rho$ of $\Gamma$.
\begin{proposition}
  \label{prop:order_ell_j}
  The integer $\lambda_j$ in the equation~\eqref{eqn:asymptotic_Rtorsion} coincides with
  the order of $\bar\rho(q_j)$.
\end{proposition}
\begin{proof}
  The order of $\rho(\ell_j)$ is equal to $2\lambda_j$. This means that 
  $\lambda_j$ is the minimum of natural numbers such that $\rho(\ell_j)^\lambda = -\I$.
  On the other hand, the order of $\bar\rho(\ell_j)$ in $\PSLC$ is the minimum of natural numbers 
  such that $\rho(\ell_j) = \pm \I$. 
  Hence the order of $\bar\rho(\ell_j)$ is equal to $\lambda_j$.
\end{proof}

We can rewrite the statement of Theorem~\ref{thm:asymptotics_Rtorsion} in terms of $\bar\rho$.
\begin{corollary}
  \label{cor:previous_result}
  If a Seifert manifold $\mfd$ has an $\SLC$-representation $\rho$ such that $\rho(h) = -\I$,
  then we have 
  \begin{equation*}
    \lim_{N \to \infty}
    \frac{\log|\Tor{\mfd}{\rho_{2N}}|
    }{2N}
    = 
    - \Big(2 -2g - \sum_{j=1}^{n} \frac{\lambda_j - 1}{\lambda_j}\Big) \log 2.
  \end{equation*}
  where $\lambda_j$ is the order of $\bar\rho(q_j)$ in $\PSLC$.
\end{corollary}

It is natural to try to start with a $\PSLC$-representation of $\Gamma$.
Here and subsequently, we focus our attention on $\PSLR$-representations of a Fuchsian group
and the induced $\SLR$-representation of $\pi_1(\mfd)$.

\medskip

Given a Fuchsian group $\Gamma = \Gamma(g; \alpha_1, \ldots, \alpha_n)$ and a $\PSLR$-representation $\bar\rho$,
we have a Seifert manifold $\mfd$ and a $\TPSLR$-representation $\widetilde{\rho}$ of $\pi_1(\mfd)$,
induced from the diagram:
\begin{equation}
\setlength{\arraycolsep}{2pt}
\begin{array}{ccccccccc}
  0 & \to & \Z=\langle h\rangle & \to &\pi_1(\mfd)& \to & \Gamma & \to & 1 \\
  &     & || &   & \quad \downarrow \raisebox{1pt}{$\scriptstyle \widetilde{\rho}$} & & \quad \downarrow \raisebox{2pt}{$\scriptstyle \bar\rho$} &  \\
  0 & \to & \Z & \to & \TPSLR & \to & \PSLR & \to & 1.
\end{array}
\end{equation}
The Seifert index $(g; (1, b), (\alpha_1, \beta_1), \ldots, (\alpha_n, \beta_n))$ of $M$ is given by 
the euler class $e(f)$.

We also have the $\SLR$-representation $\rho$ of $\pi_1(M)$ defined by the composition with the projection 
$\TPSLR \to \SLR$. 
Here we identify the universal cover $\TSLR$ with $\TPSLR$.
\begin{lemma}
  \label{lemma:h_to_shift_1}
  Suppose that 
  $(b; \beta_1, \ldots, \beta_n)$ satisfies the criteria of Theorem~\ref{thm:criterion_JN}.
  Then there exist $\TPSLR$-representations $\trho$ of $\pi_1(\Seifert)$
  such that $\trho(h) = \shift(1)$.
\end{lemma}
\begin{remark}
  Here for any $\gamma \in \R$ we write $\shift(\gamma)$ for the shift by $\gamma$,
  that is the self--homeomorphism of $\R$, $r \mapsto r + \gamma$.
  We can consider $\TPSLR$ as a subgroup of the group of homeomorphisms $f: \R \to \R$
  which are lifts of homomeomorphisms of the circle.
  The shift $\shift(\gamma)$ is an element in $\TPSLR$,
  which projects to
  $
  \begin{pmatrix}
    \cos (2\pi \gamma) & -\sin (2\pi \gamma) \\
    \sin (2\pi \gamma) & \cos (2\pi \gamma)
  \end{pmatrix}
  $
  in $\PSLR$ by the projection.
  The center of $\TPSLR$ is $\{\shift(k) \,|\, k \in \Z\} \simeq \Z$.
\end{remark}
\begin{proof}
  This is a consequence of Theorems~\ref{thm:criterion_JN} and~\ref{thm:components_fuchsian}.
\end{proof}

\begin{theorem}
  \label{thm:main_I}
  Suppose that 
  $(b; \beta_1, \ldots, \beta_n)$ satisfies the criterion of Theorem~\ref{thm:criterion_JN}.
  If $\bar\rho$ is a $\PSLR$-representation of $\Gamma$ in
  $e^{-1}(bx_0 + \beta_1 x_1 + \cdots + \beta_n x_n)$
  and $\rho$ is the induced $\SLR$-representation of $\Seifert$ by $\bar\rho$, then
  $\lambda_j$ in the asymptotic behavior~\eqref{eqn:asymptotic_Rtorsion} is expressed as
  \[
  \lambda_j = \frac{\alpha_j}{(\alpha_j, \beta_j)}
  \]
  where $(\alpha_j, \beta_j)$ is the greatest common divisor of $\alpha_j$ and $\beta_j$.

  If all $(\alpha_j, \beta_j)$ are equal to $1$, then we have
    \begin{equation*}
    \lim_{N \to \infty}
    \frac{\log|\Tor{\Seifert}{\rho_{2N}}|
    }{2N}
    = 
    - \Big(2 -2g - \sum_{j=1}^{n} \frac{\alpha_j - 1}{\alpha_j}\Big) \log 2
    =
    - \chi \log 2
    \end{equation*}
    where $\chi$ is the Euler characteristic of the base orbifold.
\end{theorem}

\begin{remark}
  If $M$ is the unit tangent bundle over a two--orbifold $\upperH^2 / \Gamma$
  whose fundamental group is embedded as a Fuchsian group $\Gamma$ in $\PSLR$,
  \ie $M = \mathrm{T}^{1} \mathbb{H}^2 / \Gamma (= \PSLR / \Gamma)$,
  then $M$ is the Seifert manifold with the index of
  $(g; (2g-2), (\alpha_1, \alpha_1-1), \ldots, (\alpha_n, \alpha_n-1))$ and
  there exists a lift $\trho$ of an embedding $\bar\rho$ of $\Gamma$ into $\PSLR$ as follows:
  \[
  \setlength{\arraycolsep}{2pt}
  \begin{array}{ccccccccc}
    0 & \to & \Z & \to &\pi_1(\mfd)& \to & \Gamma & \to & 1 \\
    &     & || &   & \quad \downarrow \raisebox{1pt}{$\scriptstyle \widetilde{\rho}$} & & \quad \downarrow \raisebox{2pt}{$\scriptstyle \bar\rho$} &  \\
    0 & \to & \Z & \to & \TPSLR & \to & \PSLR & \to & 1.
  \end{array}
  \]
  We can regard this $\trho$ as an embedding of $\pi_1(\mfd)$ into $\Isom^+(\TPSLR)$
  since $\TPSLR$ is a subgroup of $\Isom^+(\TPSLR)$.
  For more details, we refer to~\cite{scott83}.
\end{remark}

\begin{corollary}
  \label{cor:unit_tangent_bundle}
  Suppose that $M$ is the quotient of $\PSLR$ by a Fuchsian group $\Gamma$ and
  $\trho$ is an embedding of $\pi_1(\mfd)$ into $\TPSLR ( \subset \Isom^+(\TPSLR) )$.
  For the $\SLR$-representation $\rho$ induced by $\trho$, the asymptotic behavior of
  the Reidemeister torsion is expressed as
  \begin{equation*}
    \lim_{N \to \infty}
    \frac{\log|\Tor{\Seifert}{\rho_{2N}}|
    }{2N}
    = 
    - \chi \log 2
  \end{equation*}
\end{corollary}

\begin{proof}
  It follows from Euclidean algorithm that $(\alpha_i, \alpha_i-1) = 1$.
  Together with Theorem~\ref{thm:main_I}, we obtain the limit in our claim. 
\end{proof}

\begin{remark}
  The Seifert manifold $\PSLR / \Gamma$ is also regarded as
  $\TPSLR / p^{-1}(\Gamma)$ where $p$ is the projection from $\TPSLR$ onto $\PSLR$.
\end{remark}

The next lemma was shown in the proof of~\cite[Theorem~$1$]{JankinsNeumannHomomorphisms}.
\begin{lemma}
  \label{lemma:q_j_to_b_j_a_j}
  Let $\bar\rho$ be a $\PSLR$-representation of $\Gamma$ in
  $e^{-1}(bx_0 + \beta_1 x_1 + \cdots + \beta_n x_n)$.
  The induced $\TPSLR$-representation $\trho$ of $\Seifert$
  satisfies that
  every $\trho(q_j)$ is conjugate to $\shift(\beta_j / \alpha_j)$
  in $\TPSLR$ for $j=1, \ldots, n$.
\end{lemma}

We enclose this section with the proof of Theorem~\ref{thm:main_I}.
\begin{proof}[Proof of Theorem~\ref{thm:main_I}]
  Our $\SLR$-representation $\rho$ of $\pi_1(\Seifert)$ is given by
  the $\TPSLR$-representation $\trho$ as follows:
  \[
  \rho : \pi_1(\Seifert) \xrightarrow{\trho} \TPSLR \to \SLR.
  \]
  By Lemmas~\ref{lemma:h_to_shift_1} and~\ref{lemma:q_j_to_b_j_a_j},
  we can see that 
  \[
  \rho(q_j) \stackrel{\scriptstyle conj.}{\sim}
  \begin{pmatrix}
    \cos \pi \beta_j / \alpha_j & - \sin \pi \beta_j / \alpha_j \\
    \sin \pi \beta_j / \alpha_j & \cos \pi \beta_j / \alpha_j
  \end{pmatrix}.
  \]
  Hence the order of $\bar\rho(g_j)$ in $\PSLR$ is equal to $\alpha_j / (\alpha_j, \beta_j)$.
  Together with Proposition~\ref{prop:order_ell_j},
  we can obtain that $\lambda_j = \alpha_j / (\alpha_j, \beta_j)$.
\end{proof}

\section{$\Z_2$-extensions of a Fuchsian group and $\SLR$-representations}
\label{sec:Z_2_extensions}
Let $\Gamma$ be a Fuchsian group of genus $g$ with
branch indices $\alpha_1, \ldots, \alpha_n$ and
$\mfd$ be a Seifert manifold with the index
$(g; (1, b), (\alpha_1, \beta_1), \ldots, (\alpha_n, \beta_n))$.
Theorem~\ref{thm:criterion_JN} shows that 
when the integers $(b, \beta_1, \ldots, \beta_n)$ corresponds to  
some euler class $e(\bar\rho)$ in $H^2(\Gamma;\Z)$, 
we have a $\TPSLR$-representation $\trho$ of $\pi_1(\mfd)$ induced by $\bar\rho$.
Then $\trho$ also induces an $\SLR$-representation $\rho$ such that $\rho(h)=-\I$.

We can also find $\SLR$-representations of $\pi_1(\mfd)$ induced by
other $(b', \beta'_1, \ldots, \beta'_n)$ in $H^2(\Gamma;\Z)$.
We classify $(b', \beta'_1, \ldots, \beta'_n)$ which gives 
an $\SLR$-representation of $\pi_1(\mfd)$.

\subsection{$\Z_2$-extension and $\SLR$-representation}
Every irreducible $\SLR$-representation $\rho$ of $\pi_1(\mfd)$ factors through
$\pi_1(\mfd) / \langle h^2 \rangle$
since $\rho$ sends the central element $h$ into the center of $\SLR$.
When we regard $\Gamma$ as $\pi_1(\mfd) / \langle h \rangle$,
we have the following central extension of $\Gamma$:
\[
0 \to \Z / 2\Z \to \pi_1(\mfd) / \langle h^2 \rangle \to \Gamma \to 1.
\]
Taking the pull--back central extension from 
\[
\setlength{\arraycolsep}{2pt}
\begin{array}{ccccccccc}
  & & & & & & \Gamma & &  \\
  & & & & & & \; \downarrow \raisebox{2pt}{$\scriptstyle \bar\rho$} &  \\
  0 & \to & \Z / 2\Z & \to & \SLR & \to & \PSLR & \to & 1,
\end{array}
\]
we have the following diagram:
\[
\setlength{\arraycolsep}{2pt}
\begin{array}{ccccccccc}
  0 & \to & \Z / 2\Z & \to & \hat\Gamma & \to & \Gamma & \to & 1 \\
  &     & || &   & \; \downarrow \raisebox{1pt}{${\scriptstyle \rho}$} & & \; \downarrow \raisebox{2pt}{$\scriptstyle \bar\rho$} &  \\
  0 & \to & \Z / 2\Z& \to & \SLR & \to & \PSLR & \to & 1.
\end{array}
\]
Here $\hat \Gamma$ is isomorphic to $\pi_1(\mfd) / \langle h^2 \rangle$ for some $\mfd$.
\begin{lemma}
  \label{lemma:ExtZ2}
  The group of central extensions of $\Gamma$ by $\Z / 2\Z$
  is expressed as 
  \[
  \Ext(\Gamma;\Z/2\Z)
  \simeq
  \Ext(\Gamma;\Z) \,/ \,2 \Ext(\Gamma;\Z).
  \]
\end{lemma}
\begin{remark}
Lemma~\ref{lemma:ExtZ2} is a consequence of
the remark following the proof of~\cite[Theorem~$10.4$ ]{NeumannJankins:Seifert}. 
\end{remark}
\begin{proof}
  In proving the surjectivity of the following homomorphism:
  \[
  H^2(\Gamma;\Z) = \Ext(\Gamma;\Z/2\Z) 
  \to
  ab\langle x_0, \ldots, x_n \,|\, \alpha_i x_i = x_0, i=1, \ldots, n \rangle,
  \]
  we define a function
  \[
  \nu : ab\langle x_0, \ldots, x_n \,|\, \alpha_i x_i = x_0, i=1, \ldots, n \rangle
  \to S(\Z) = \{\hbox{subgroups of $\Z$}\}
  \]
  by $[s_0, \ldots, s_n] \mapsto \ker(\Z \to \pi (s_0, \ldots, s_n))$.
  Here $\pi(s_0, \ldots, s_n)$ is a central extension of $\Gamma$ given by
  $\pi_1(\mfd(g;(1,s_0), (\alpha_1, s_1), \ldots, (\alpha_n, s_n)))$.
  The function $\nu$ gives a central extension of $\Gamma$ by
  \[
  1 \to \Z / \nu([s_0, \ldots, s_n]) \to \pi(s_0, \ldots, s_n) \to \Gamma \to 1
  \]
  for each $[s_0, \ldots, s_n]$.
  It is shown in~\cite[the proof of Theorem~$10.4$ ]{NeumannJankins:Seifert} that
  \begin{equation}
    \label{eqn:proof_ext_Z}
    \Ext(\Gamma; \Z) \simeq \nu^{-1}(\{0\}), \quad 
    \nu^{-1}(\{0\}) = ab\langle x_0, \ldots, x_n \,|\, \alpha_i x_i = x_0, i=1, \ldots, n \rangle.
  \end{equation}

  In the case of $\Ext(\Gamma;\Z / 2\Z)$,
  we also define a function
  \[
  \nu' :
  ab\langle x_0, \ldots, x_n \,|\, \alpha_i x_i = x_0, i=1, \ldots, n \rangle
  \to S(\Z) = \{\hbox{subgroups of $\Z$}\}
  \]
  by $[s_0, \ldots, s_n] \mapsto \ker(\Z \to \pi (s_0, \ldots, s_n) / \langle h^2 \rangle)$.
  One can see that 
  \[
  \Ext(\Gamma; \Z / 2\Z)
  \simeq \nu'^{-1}(2\Z) \,/\, 2 \, ab\langle x_0, \ldots, x_n \,|\, \alpha_i x_i = x_0, i=1, \ldots, n \rangle
  \]
  since the submodule $2 \, ab\langle x_0, \ldots, x_n \,|\, \alpha_i x_i = x_0, i=1, \ldots, n \rangle$
  corresponds to the trivial extension.
  We can see that each element of $\nu^{-1}(\{0\})$ is contained in $\nu'^{-1}(2\Z)$
  by the following diagram:
  \[
  \setlength{\arraycolsep}{2pt}
  \begin{array}{ccccccccc}
    1 & \to & \Z & \to & \pi(s_0, \ldots, s_n) & \to & \Gamma & \to & 1 \\
    &     &\downarrow &     & \downarrow & \nearrow &  &  \\
    & & \Z/2\Z & \to & \pi(s_0, \ldots, s_n)/\langle h^2 \rangle. &  &  & & 
  \end{array}
  \]  
  Hence it follows from
  \[
  \nu^{-1}(\{0\}) \subset
  \nu'^{-1}(2\Z) \subset
  ab\langle x_0, \ldots, x_n \,|\, \alpha_i x_i = x_0, i=1, \ldots, n \rangle \\
  \]
  and \eqref{eqn:proof_ext_Z}
  that
  $
  \nu^{-1}(2\Z) =
  ab\langle x_0, \ldots, x_n \,|\, \alpha_i x_i = x_0, i=1, \ldots, n \rangle
  \simeq \Ext(\Gamma;\Z).
  $
  Therefore we have the isomorphism between $\Ext(\Gamma;\Z/2\Z)$ and $\Ext(\Gamma;\Z) / 2 \Ext(\Gamma;\Z)$.
\end{proof}

\begin{theorem}
  \label{prop:Z2extension_representation}
  Suppose that 
  both of $(b, \beta_1, \ldots, \beta_n)$ and $(b', \beta'_1, \ldots, \beta'_n)$
  give the same class in $\Ext(\Gamma;\Z/2\Z)$.
  If $(b', \beta'_1, \ldots, \beta'_n)$ gives an euler class as in Theorem~\ref{thm:criterion_JN}, 
  then 
  the central extension $[(b', \beta'_1, \ldots, \beta'_n)]$ in $\Ext(\Gamma;\Z)$
  induces an $\SLR$-representation $\rho'$ of $\pi_1(\mfd)$ such that $\rho'(h) = -\I$.
\end{theorem}

\begin{proof}
  It follows from the assumption that we have an isomorphism $\varphi$
  from $\pi_1(M) / \langle h^2 \rangle$ to $\pi_1(M') / \langle h'^2 \rangle$ as
  \[
  \setlength{\arraycolsep}{2pt}
  \begin{array}{ccccccccc}
    0 & \to & \Z / 2\Z & \to & \pi_1(M) / \langle h^2 \rangle & \to & \Gamma & \to & 1 \\
    &     & || &   & \quad \downarrow \raisebox{1pt}{${\scriptstyle \varphi}$} & & || &  \\
    0 & \to & \Z / 2\Z& \to & \pi_1(M') / \langle h'^2 \rangle & \to & \Gamma & \to & 1
  \end{array}
  \]
  where $M$ and $M'$ are the Seifert manifolds with the indices
  $(g; (1, b), (\alpha_1, \beta_1), \ldots, (\alpha_n, \beta_n))$ and
  $(g; (1, b'), (\alpha_1, \beta'_1), \ldots, (\alpha_n, \beta'_n))$.
  Since $b' x_0 + \beta'_1 x_1 + \cdots + \beta'_n x_n$ is an euler class in $H^2(\Gamma;Z) = \Ext(\Gamma;\Z)$,
  there exists a $\TPSLR$-representation $\trho'$ of $\pi_1(M')$.
  This $\TPSLR$-representation $\trho'$ gives an $\SLR$-representation $\rho'$ such that $\rho'(h') = -\I$.
  Taking the pull--back of the homomorphism from $\pi_1(M') / \langle h'^2 \rangle$ to $\SLR$ by $\varphi$,
  we obtain a homomorphism from $\pi_1(M) / \langle h^2 \rangle$ to $\SLR$.
  The composition with the projection from $\pi_1(M)$ gives an $\SLR$-representation of $\pi_1(M)$
  sending $h$ to $-\I$.
\end{proof}

\begin{remark}
  Although $(b, \beta_1, \ldots, \beta_n)$ is contained in
  the equivalent class of $(b', \beta'_1, \ldots, \beta'_n)$ in $\Ext(\Gamma;\Z/2\Z)$,
  the induced representation $\rho$ by $(b, \beta_1, \ldots, \beta_n)$
  is not necessarily conjugate to $\rho'$ induced by
  $(b', \beta'_1, \ldots, \beta'_n)$.
  We can see an example in Section~\ref{sec:examples}.
\end{remark}

\begin{proposition}
  \label{prop:euler_class_conjugacy_class}
  Suppose that
  $\bar\rho$ and $\bar\rho'$ are $\PSLR$-representations of $\Gamma$
  satisfying that
  $[e(\bar\rho)] = [e(\bar\rho')] = [b x_0 + \beta_1 x_1 + \cdots + \beta_n x_n]$
  in $\Ext(\Gamma;\Z/2\Z)$
  but $e(\bar\rho) \not = e(\bar\rho')$ in $\Ext(\Gamma;\Z)$.

  Then the $\SLR$-representation $\rho$ of $\pi_1(\Seifert)$ is not conjugate to
  $\rho'$.
\end{proposition}

\begin{proof}
  If $\rho$ were conjugate to $\rho'$,
  then $\bar\rho$ would be also conjugate to $\bar\rho'$.
  Since the euler class is invariant under conjugation
  (we refer to the proof of~\cite[Theorem~$1$]{JankinsNeumannHomomorphisms}),
  the euler class $e(\bar\rho)$ must coincide with $e(\bar\rho')$.
  This is a contradiction to the assumption that $e(\bar\rho) \not = e(\bar\rho')$.
\end{proof}

\begin{theorem}
  \label{thm:asymptotic_another_SLR_rep}
  Suppose that $M$ denotes a Seifert manifold $\Seifert$ and 
  integers $(b', \beta'_1, \ldots, \beta'_n)$ gives an euler class
  satisfying the criteria of Jankins and Neumann in Theorem~\ref{thm:criterion_JN}.
  If $(b, \beta_1, \ldots. \beta_n)$ and $(b', \beta'_1, \ldots, \beta'_n)$ give
  the same class in $\Ext(\Gamma;\Z/2\Z)$, then
  there exists an $\SLR$-representation $\rho'$ of $\pi_1(\mfd)$ such that
  the asymptotic behavior of $\Tor{\mfd}{\rho'_{2N}}$ is expressed as
  \[
  \lim_{N \to \infty}
  \frac{\log|\Tor{\mfd}{\rho'_{2N}}|
  }{2N}
  = 
  - \Big(2 -2g - \sum_{j=1}^{n} \frac{\lambda'_j - 1}{\lambda'_j}\Big) \log 2
  \]
  where $\lambda'_j = \alpha_j / (\alpha_j, \beta'_j)$.
\end{theorem}

\section{Examples}
\label{sec:examples}
\subsection{$\SUII$-representations of a Brieskorn manifold}
It is known that $\SLR$ is conjugate to 
$\SUII
= \left\{\left.
\begin{pmatrix}
  \xi & \eta \\
  \bar \eta & \bar \xi
\end{pmatrix}\,\right|\,
|\xi|^2 - |\eta|^2 = 1
\right\}$
by
$\begin{pmatrix}
  1 & -\sqrt{-1} \\
  1 & \sqrt{-1}
\end{pmatrix}$.
We can consider $\SUII$-representations 
instead of $\SLR$.

Suppose that $\mfd$ is a Seifert homology sphere with three exceptional fibers.
The genus of the base orbifold must be zero. The Seifert index is given by 
\[
(0; (1, b), (\alpha_1, \beta_1), (\alpha_2, \beta_2), (\alpha_3, \beta_3)).
\]
We can express $\pi_1(\mfd)$ as
\[
\pi_1(\mfd) =
\langle
q_1, q_2, q_3, h \,|\,
\hbox{$h$:central},
q_1^{\alpha_j} = h^{\beta_j} (j=1, 2, 3), q_1 q_2 q_3 = h^{-b}
\rangle.
\]
We write $\begin{pmatrix}
  \xi_j & \eta_j \\
  \bar\eta_j & \bar\xi_j 
\end{pmatrix}$
for $\rho(q_j)$ 
and $a_j + b_j \sqrt{-1}$
for each $\xi_j$.

We will compute irreducible $\SUII$-representations of $\pi_1(\mfd)$
such that $\rho(h) = -\I$ up to conjugation.
\begin{definition}
  \label{def:triple_for_rep}
  For an irreducible $\SUII$-representation of $\pi_1(\mfd)$,
  we have the triple $(k_1, k_2, k_3)$ of natural numbers such that 
  \begin{itemize}
  \item $\trace \rho(q_j) = 2 \cos (k_j \pi / \alpha_j)$;
  \item $0 < k_j < \alpha_j$;
  \item $k_j \equiv \beta_j \bmod 2$.
  \end{itemize}
\end{definition}

\begin{lemma}
  \label{lemma:conj_class_SUII}
  Let $\rho$ be an irreducible $\SUII$-representation of $\pi_1(\mfd)$.
  The conjugacy class of $\rho$ is determined by 
  the pair of the triple $(k_1, k_2, k_3)$ and the sign of $b_1$.
\end{lemma}

\begin{proof}
  By the relations of $\pi_1(\mfd)$, the representation of $\rho$ is determined by
  $\xi_1$, $\eta_1$, $\xi_2$ and $\eta_2$.
  We can assume that $\rho(q_1)$ is diagonal \ie $\eta_1 =0$, and $\eta_2$ is a positive real number.
  The second assumption is realized by the conjugation of a diagonal matrix.
  Then $a_1$ equals to $\cos (k_1 \pi / \alpha_1)$ and $b_1 = \pm \sin (k_1 \pi / \alpha_1)$. 
  Since $\trace \rho(q_2)$ equals to $2a_2$, we also have that $a_2 = \cos (k_2 \pi / \alpha_2)$.
  By the relation that $q_1q_2 = h^{-b} q_3^{-1}$, we have that 
  \[
  \trace \rho(q_3) = \trace \rho(q_3)^{-1} = (-1)^b 2 (a_1 a_2 - b_1 b_2).
  \]
  Together with $\trace \rho(q_3) = 2\cos (k_3 \pi / \alpha_3)$, we can see that $b_2$ is determined
  by $(k_1, k_2, k_3)$ and the sign of $b_2$.
  Last $\eta_2$  is given by the equality that $a_2^2 + b_2^2 - \eta_2^2 = 1$.
\end{proof}

\begin{remark}
  The diagonal matrix
  $\begin{pmatrix}
    e^{\sqrt{-1} \theta} & 0 \\
    0 & e^{-\sqrt{-1} \theta} 
  \end{pmatrix}$
  is not conjugate to 
  $\begin{pmatrix}
    e^{-\sqrt{-1} \theta} & 0 \\
    0 & e^{\sqrt{-1} \theta} 
  \end{pmatrix}$
  in $\SUII$.
\end{remark}

\begin{proposition}
  Suppose that an $\SUII$-representation $\rho$ of $\pi_1(\mfd)$ is irreducible and
  satisfies that $\rho(h) = -\I$.
  
  If $b$ is even, then the corresponding triple $(k_1, k_2, k_3)$ satisfies the either inequality:
  \begin{equation}
    \label{eqn:constrain_b_even}
    0 < \frac{k_3}{\alpha_3} \leq \left| \frac{k_1}{\alpha_1} - \frac{k_2}{\alpha_2} \right|
    \quad \text{or} \quad
    1 - \left| \frac{k_1}{\alpha_1} + \frac{k_2}{\alpha_2} - 1 \right| \leq \frac{k_3}{\alpha_3} < 1.
  \end{equation}

  If $b$ is odd, then the corresponding triple $(k_1, k_2, k_3)$ satisfies the either inequality:
  \begin{equation}
    \label{eqn:constrain_b_odd}
    0 < \frac{k_3}{\alpha_3} \leq \left| \frac{k_1}{\alpha_1} + \frac{k_2}{\alpha_2} - 1 \right|
    \quad \text{or} \quad
    1 - \left| \frac{k_1}{\alpha_1} - \frac{k_2}{\alpha_2}\right| \leq \frac{k_3}{\alpha_3} < 1.
  \end{equation}
\end{proposition}

\begin{proof}
  We prove the case that $b$ is even.
  In this case, we have the equality that $\rho(q_3) = \rho(q_2)^{-1}\rho(q_1)^{-1}$. 
  We can assume
  $\rho(q_1)$ is diagonal, \ie
  \[
  \rho(q_1)=
  \begin{pmatrix}
    \xi_1 & 0 \\
    0 & \bar \xi_1
  \end{pmatrix}
  \]
  without loss of generality.
  Hence the trace $\trace \rho(q_3) = 2 \cos(k_3 \pi / \alpha_3)$ equals to 
  $\xi_1 \xi_2 + \bar\xi_1 \bar\xi_2 = 2(a_1 a_2 - b_1 b_2)$.
  We have that 
  \begin{equation}
    \label{eqn:cos_el_3}
    a_1 a_2 - b_1 b_2 = \cos \Big(\frac{k_3 \pi}{\alpha_3}\Big).
  \end{equation}
  It follows from the assumption that $a_1 = \cos(k_1 \pi / \alpha_1)$, $b_1 = \pm \sin (k_1 \pi / \alpha_1)$.
  We also have the inequality that $b_2^2 \geq \sin^2 (k_2 \pi / \alpha_2)$
  from 
  $|\xi_2|^2 = 1 + |\eta_2|^2 \geq 1$,
  $|\xi_2|^2 = a_2^2 + b_2^2$ and $a_2 = \cos(k_2 \pi / \alpha_2)$.
  Together with Eq.~\eqref{eqn:cos_el_3}, we obtain the following constrains:
  \begin{gather*}
    \cos\Big(k_1 \pi / \alpha_1\Big)\cos\Big(k_2 \pi / \alpha_2\Big)
    \mp \sin\Big(k_1 \pi / \alpha_1\Big) b_2
    \geq \cos\Big(k_3 \pi / \alpha_3\Big), \\
    b_2^2 \geq \sin^2\Big(k_2 \pi / \alpha_2\Big).
  \end{gather*}
  We can derive our inequality~\eqref{eqn:constrain_b_even} from the above constrains.
  Similarly we can also derive the inequality~\eqref{eqn:constrain_b_odd} in the case that $b$ is odd.
\end{proof}

\begin{remark}
  One can find the similar inequality for $\SU$-representations
  in~\cite{FintushelStern:InstantonSeifert}, \cite[\S$14.5$]{Saveliev99:LectureTopology3mfd}.
\end{remark}

\subsection{Brieskorn manifold of type $(2, 3, 7)$}
Let $M$ be the Seifert manifold of the index
$(g; (1, b), (\alpha_1, \beta_1), (\alpha_2, \beta_2), (\alpha_3, \beta_3))
= (0; (1, -1), (2, 1), (3, 1), (7, 1))$.
Then the fundamental group $\pi_1(\mfd)$ is expressed as
\[
\pi_1(\mfd) =
\langle
q_1, q_2, q_3, h \,|\,
\hbox{$h$:central},
q_1^2 = h, q_3^3 = h, q_3^7 = h, q_1 q_2 q_3 = h^{-(-1)}
\rangle.
\]

\begin{lemma}
  There are two conjugacy classes of irreducible $\SUII$-representations $\rho$
  of $\pi_1(\mfd)$ such that $\rho(h)=-\I$.
\end{lemma}

\begin{proof}
  We have the triple $(k_1, k_2, k_3)$ of natural numbers
  for the $\SUII$-representation $\rho$, as in Definition~\ref{def:triple_for_rep}.
  This triple satisfies the constrain~\eqref{eqn:constrain_b_odd}.
  Therefore we have only one triple $(k_1, k_2, k_3) = (1, 1, 1)$.
  Since we have two possibility of the sign of $b_1$, we can conclude that
  there are the two conjugacy classes of irreducible $\SUII$-representations
  of $\pi_1(\mfd)$.
\end{proof}

We see the correspondence between the conjugacy classes and
the euler classes of $\PSLR$-representations of $\Gamma$.
From Theorem~\ref{prop:Z2extension_representation}
it is enough to find equivalent classes $[(b', \beta'_1, \ldots, \beta'_3)]$
in $\Ext(\Gamma;\Z/2\Z)$, which satisfy the criteria of
Theorem~\ref{thm:criterion_JN}.
Since the genus $g$ in the Seifert index is equal to $0$ and
the number $n$ of exceptional fibers is equal to $3$, 
we consider the cases~\eqref{item:b_eq_minus_one} and~\eqref{item:b_eq_one_minus_n}
in Theorem~\ref{thm:criterion_JN}.

\begin{itemize}
\item[\eqref{item:b_eq_minus_one}] We suppose that $b'=-1$.
  If we find $(\beta'_1, \beta'_2, \beta'_3)$ such that
  $\beta'_1 / 2 + \beta'_2 / 3 + \beta'_3 / 7 \leq 1$,
  then we have only solution $(1, 1, 1)$.
  Since these integers $(b', \beta'_1, \beta'_2, \beta'_3$ coincides with
  $(b, \beta_1, \beta_2, \beta_3) = (-1, 1, 1, 1)$ in the Seifert index of $\mfd$,
  this solution gives an $\TPSLR$-representation of $\pi_1(\mfd)$.
\item[\eqref{item:b_eq_one_minus_n}]
  We suppose that $b'=-2$.
  If we find $(\beta'_1, \beta'_2, \beta'_3)$ such that
  $\beta'_1 / 2 + \beta'_2 / 3 + \beta'_3 / 7 \geq n-1 = 2$,
  then we have only solution $(1, 2, 6)$.
  This solution induces an $\TPSLR$-representation of $\pi_1(\mfd')$
  where $M'$ is the Seifert manifold with the index of
  \[
  (g; (1, b), (\alpha_1, \beta_1), (\alpha_2, \beta_2), (\alpha_n, \beta_n))
  = (0; (1, -2), (2, 1), (3, 2), (7, 6)).
  \]
\end{itemize}

\begin{remark}
  \label{remark:isom_manifolds}
  The Seifert manifold $\mfd'$ in~\eqref{item:b_eq_one_minus_n} is
  obtained by the homeomorphism by reversing the orientation of a fiber in $M$.
  This is due to that 
  the Seifert index of $\mfd'$ is obtained from $\mfd$ as follows:
  \begin{align*}
    (0; (1, -1), (2, 1), (3, 1), (7, 1))
    &\xrightarrow{ori.~rev.}
    (0; (1, 1), (2, -1), (3, -1), (7, -1)) \\
    &=
    (0; (1, 1-3), (2, -1+2), (3, -1+3), (7, -1+7)) \\
    &=
    (0; (1, -2), (2, 1), (3, 2), (7, 6)).
  \end{align*}
  We refer to~\cite[Theorem~$1.5$]{NeumannJankins:Seifert}
  for the above operations of Seifert index.
\end{remark}

\begin{lemma}
  The isomorphism from $\pi_1(\mfd)$ to $\pi_1(\mfd')$
  induced by the homeomorphism in Remark~\ref{remark:isom_manifolds}.
  is given by the following correspondence:
  \[
    h \mapsto h'^{-1}, \quad
    q_j \mapsto q'_j h'^{-1}.
  \] 
\end{lemma}

\begin{proof}
  The homeomorphism by reversing the orientation of a fiber induces
  the correspondence sending $h$ to $h'^{-1}$.
  The change of Seifert index from $(0; (1, 1), (2, -1), (3, -1), (7, -1))$
  to $(0; (1, -2), (2, 1), (3, 2), (7, 6))$ corresponds to 
  changing the generators $q_j$ to $q_j h'_j$.
  If we write $q'_j$ for the new generators, then we have the presentation of $\pi_1(\mfd')$.
\end{proof}

\begin{proposition}
  Let $\rho$ and $\rho'$
  be irreducible $\SUII$-representations of $\pi_1(\mfd)$
  induced by $(b, \beta_1, \beta_2, \beta_3) = (-1, 1, 1, 1)$
  and $(b', \beta'_1, \beta'_2, \beta'_3) = (-2, 1, 1, 1)$.
  The conjugacy class of $\rho$ is different from that of $\rho'$.
\end{proposition}

\begin{proof}
  The euler class of $e(\bar \rho')$ is equal to $- e(\bar \rho)$.
  It follows from Proposition~\ref{prop:euler_class_conjugacy_class}
  that $\rho$ is not conjugate to $\rho'$.
\end{proof}

\begin{corollary}
  The $\SUII$-representations $\rho$ and $\rho'$ give all representatives in
  the conjugacy classes of irreducible $\SUII$-representations sending $h$ to $-\I$.
\end{corollary}

Therefore we have the isomorphism $\varphi$ from $\pi_1(\mfd)$ to $\pi_1(\mfd')$
induced by the orientation reversing homeomorphism.
The composition of $\varphi$ gives an $\TPSLR$-representation of $\pi_1(\mfd)$.
These two $\TPSLR$-representations give the representatives of 
different conjugacy classes in the set of $\SLR$-representations of $\pi_1(\mfd)$.

This observation can be extend to a general Seifert manifold.
\begin{theorem}
  Suppose that $M$ is a Seifert manifold and $M'$ is the orientation reversed manifold $-M$ along fibers.
  If there exists a $\TPSLR$-representation $\trho$ of $\pi_1(M)$, then
  we also have a $\TPSLR$-representation $\trho'$ of $\pi_1(M')$ such that
  \begin{itemize}
  \item $\trho'$ induces an $\SLR$-representation $\rho'$ of $\pi_1(M)$
    which arises the following asymptotic behavior of the Reidemeister torsion:
    \[
    \lim_{N \to \infty}
    \frac{\log|\Tor{\Seifert}{\rho'_{2N}}|
    }{2N}
    =
    \lim_{N \to \infty}
    \frac{\log|\Tor{\Seifert}{\rho_{2N}}|
    }{2N}
    \]
    where $\rho$ is the $\SLR$-representation induced by $\trho$;
  \item $e(\bar\rho') = -e(\bar \rho)$
    where $\bar \rho$ and $\bar \rho'$ are the induced $\PSLR$-representation of $\Gamma$.
  \end{itemize}
\end{theorem}

\begin{proof}
  Let $\varphi$ denote the orientation reversing homeomorphism from $M$ to $M'$.
  Then we can choose the composition $\rho \circ \varphi^{-1}$ as $\rho'$.
  By the construction,
  we can see the equality on the asymptotic behaviors of the Reidemeister torsions.
  We also have the equality on the euler classes of $\PSLR$-representations of $\Gamma$.
\end{proof}

\section*{Acknowledgment}
The author wishes to express his thanks to Professor Makoto Sakuma for drawing the author's attention to 
$\PSLR$-representations of a Fuchsian group and the induced asymptotic behavior of
the Reidemeister torsion.
This research was supported by 
by JSPS KAKENHI Grant Number $26800030$.
%%%%%%%%%%%%%%%%%%%%%%%%%%%%%%%%%%%%%%%%%%%%%%%%%%%%%%%%%%%%%%%%%%%% 
% reference
%%%%%%%%%%%%%%%%%%%%%%%%%%%%%%%%%%%%%%%%%%%%%%%%%%%%%%%%%%%%%%%%%%%% 
\bibliographystyle{amsalpha}
\bibliography{torsionSLgeometry}

\end{document}